\documentclass{article}
\usepackage{amsmath,amsthm,amssymb,amscd}

\newcommand{\ga}{\alpha}
\newcommand{\gb}{\beta}

\newcommand{\gd}{\delta}
\newcommand{\gw}{\omega}

\newcommand{\gs}{\sigma}
\newcommand{\eps}{\varepsilon}

\newcommand{\coll}{\mathrm{Coll}}

\newcommand{\cantor}{2^\gw}

\newcommand{\supp}{\mathrm{supp}}

\newcommand{\dom}{\mathrm{dom}}

\newcommand{\sing}{\mathrm{Sing}}
\newcommand{\regg}{\mathrm{Reg}}
\newcommand{\colors}{\mathrm{Col}}

\newtheorem{theorem}{Theorem}[section]

\newtheorem{claim}[theorem]{Claim}
\newtheorem{corollary}[theorem]{Corollary}
\newtheorem{fact}[theorem]{Fact}
\newtheorem{proposition}[theorem]{Proposition}

\theoremstyle{definition}
\newtheorem{definition}[theorem]{Definition}
\newtheorem{example}[theorem]{Example}
\newtheorem{question}[theorem]{Question}

\title{Coloring the distance graphs\footnote{2020 AMS subject classification 03E35, 14P99, 05C15.} }

\author{
Jind{\v r}ich Zapletal\\
University of Florida\\
zapletal@ufl.edu\\
ORCID 0000-0003-3437-5073}

\begin{document}
\maketitle

\begin{abstract}
Let $n\geq 1$ be a number. Let $\Gamma_n$ be the graph on $\mathbb{R}^n$ connecting points of rational Euclidean distance. It is consistent with choiceless set theory ZF+DC that $\Gamma_n$ has countable chromatic number, yet the chromatic number of $\Gamma_{n+1}$ is uncountable.
\end{abstract}

\section{Introduction}

Let $n\geq 1$ be a number. Let $\Gamma_n$ be the graph on $\mathbb{R}^n$ connecting points of rational Euclidean distance. Komj{\' a}th \cite{komjath:decomposition} proved  that in ZFC, all graphs $\Gamma_n$ have countable chromatic number; the cases $n=2$ and $n=3$ are easier and have been known much earlier \cite{komjath:r3} \cite{hajnal:chromatic}. Difficulty of the proofs greatly increases with $n$.  The main theorem of this paper shows why this is so.

\begin{theorem}
\label{maintheorem}
Let $n\geq 1$ be a number. The statement ``the chromatic number of $\Gamma_n$ is countable while that of $\Gamma_{n+1}$ is not"  is consistent with ZF+DC relative to an inaccessible cardinal.
\end{theorem}

\noindent The cases of $n\leq 3$ have been resolved previously by somewhat ad hoc methods. The case $n=1$ is \cite[Corollary 12.3.16]{z:geometric}, $n=2$ is proved in \cite[Corollary 12.3.18]{z:geometric}, and the harder case $n=3$ is proved in \cite{z:distance3}. Unsurprisingly, Theorem~\ref{maintheorem} is a special case of a much stronger result. 

\begin{theorem}
\label{maintheorem2}
Let $n\geq 1$ be a number. Let $\Gamma$ be a $\gs$-algebraic hypergraph on $\mathbb{R}^n$ containing no perfect clique. From an inaccessible cardinal there is a model of ZF+DC in which the chromatic number of $\Gamma$ is countable, while in every non-meager subset of $\mathbb{R}^{n+1}$ it is possible to find points of every small enough distance.
\end{theorem}

\noindent Here, a $\gs$-algebraic graph on $\mathbb{R}^n$ is one for which there are countably many polynomials $\phi_m(u, v)$ for $m\in\gw$, each with real coefficients and $2n$ many free variables such that any distinct points $x, y\in\mathbb{R}^n$ are connected if there is $m$ such that $\phi_m(x, y)=0$. For such graphs, in ZFC nonexistence of perfect clique is equivalent to countable chromatic number \cite{schmerl:avoidable}; the colorings are obtained through heavy use of the Axiom of Choice. Theorem~\ref{maintheorem2} is clearly the strongest possible result in a certain direction. It follows from a conjunction of the coloring poset construction in Theorem~\ref{posettheorem} and a preservation result in Theorem~\ref{preservationtheorem}.

The area of chromatic numbers of algebraic and $\gs$-algebraic graphs and hypergraphs in choiceless context offers rich and novel interplay between forcing, combinatorics, and real algebraic geometry. Among many attractive combinatorial problems which are left untouched by the results of this paper, I will state only the following attractive question concerning separation of chromatic numbers in the same dimension.

\begin{question}
For $n\geq 1$ and a countable set $a$ of positive real numbers, let $\Gamma_{na}$ be the graph on $\mathbb{R}^n$ connecting points whose distance belongs to $a$. Characterize the pairs $\langle a, b\rangle$ such that it is consistent with ZF+DC that the chromatic number of $\Gamma_{na}$ is countable while that of $\Gamma_{nb}$ is not.
\end{question}

\noindent The proof of Theorem~\ref{maintheorem2} uses the approach of geometric set theory \cite{z:geometric}. The model is constructed as a generic extension of the classical choiceless Solovay model \cite[Theorem 26.14]{jech:newset} by a rather canonical coloring poset (Definition~\ref{posetdefinition}).  The method uses an inaccessible cardinal to support the general framework. Removing the inaccessible cardinal in the spirit of \cite{Sh:176} requires plenty of improvisation, but in the given case is probably possible. Notation of the paper uses the set theoretic standard of \cite{jech:newset}, and in matters of geometric set theory, \cite{z:geometric}. DC is the Axiom of Dependent Choices. All theorems and definitions take place in ZFC set theory. 

\section{Mutually Noetherian extensions}

This section introduces the main technical notion connecting dimension of Euclidean spaces with generic extensions. To start, recall the following standard definitions of algebraic geometry.

\begin{definition}
Let $\langle X, \mathcal{T}\rangle$ be a topological space.

\begin{enumerate}
\item $\mathcal{T}$ is \emph{Noetherian} if there is no infinite sequence of $\mathcal{T}$-closed sets strictly decreasing with respect to inclusion.
\item  A $\mathcal{T}$-closed set is \emph{irreducible} if it is not the union of finitely many properly smaller $\mathcal{T}$-closed subsets.
\item The \emph{Krull dimension} of $\mathcal{T}$ is the maximum length of chains of irreducible closed sets linearly ordered by inclusion, minus one. if the maximum does not exist, the Krull dimension is infinite.
\end{enumerate}
\end{definition}

\noindent Note that this definition provides only for finite and not transfinite values of Krull dimension, which is fine for the purposes of this paper. Noetherian topologies are commonly identified by their closed sets. The most basic example is the topology of algebraic subsets of $\mathbb{R}^n$ for a number $n\geq 1$, which has Krull dimension $n$ by the Hilbert basis theorem. The following definitions connect Noetherian topologies with descriptive set theory and forcing.

\begin{definition}
Let $X$ be a $K_\gs$ Polish space. A topology $\mathcal{T}$ on $X$ is \emph{analytic} if every $\mathcal{T}$-closed set is closed in the Polish topology, and the collection of all $\mathcal{T}$-closed sets is an analytic subset of the standard Borel space $F(X)$ of closed subsets of $X$.
\end{definition}

\noindent For an analytic topology on a $K_\gs$ space $X$, standard complexity and Shoenfield absoluteness arguments show that the properties such as Noetherian status, irreducible status of closed sets, and the Krull dimension are all absolute between transitive models of ZFC containing all ordinals.

\begin{definition}
Let $X$ be a Polish space and $\mathcal{T}$ be a Noetherian topology on $X$. Let $M$ be a transitive model of set theory containing the code for $X$, and let $A\subset X$ be a set. The symbol $C(M, A)$ denotes the inclusion-smallest $\mathcal{T}$-closed set coded in $M$ which contains $A$ as a subset.
\end{definition}

\noindent Note that the set $C(M, A)$ is well-defined since the search for ever smaller $\mathcal{T}$-closed sets coded in $M$ containing $A$ as a subset cannot go on forever by the Noetherian property of $\mathcal{T}$. The dependence of the set $C(M, A)$ on the topology is suppressed as $\mathcal{T}$ is always clear from the context.

\begin{definition}
\label{nnoetheriandefinition}
Let $n\geq 1$ be a natural number. For generic extensions $V[G_0], V[G_1]$, say that $V[G_0]$ is $n$-\emph{Noetherian over} $V[G_1]$ if for every Polish $K_\gs$-space $X$ and every analytic Noetherian topology $\mathcal{T}$ of Krull dimension smaller than $n$ coded in the ground model, and every set $A\subset X$ in $V[G_1]$, $C(V, A)=C(V[G_0], A)$. If in addition$V[G_1]$ is $n$-\emph{Noetherian over} $V[G_0]$, we call these generic extensions \emph{mutually $n$-Noetherian}.
\end{definition}

\begin{proposition}
\label{enoughproposition}
In Definition~\ref{nnoetheriandefinition}, it is only necessary to consider irreducible $\mathcal{T}$-closed sets $A\subset X$.
\end{proposition}

\begin{proof}
 Given an arbitrary set $A\subset X$ in the model $V[G_1]$, working in $V[G_0, G_1]$ find the smallest $\mathcal{T}$-closed set $C\subset X$ containing $A$ as a subset, and write it as $C=\bigcup_{i\in j}C_i$ as a finite union of irreducible $\mathcal{T}$-closed sets with smallest possible $j$. 

Observe is that the sets $C$ as well as $C_j$ for $i\in j$ are coded in the model $V[G_1]$. To see this, writing $\bar A$ for the Polish closure of $A$, $C$ is the smallest $\mathcal{T}$-closed set containing $\bar A$ as a subset, and $V[G_1]$ evaluates this set correctly by a Shoenfield absoluteness argument. The model $V[G_1]$ finds a decomposition of $C$ into irreducible $\mathcal{T}$-closed sets. By the uniqueness of this decomposition applied in $V[G_0, G_1]$, it is clear that $C=\bigcup_iC_i$ is exactly the decomposition that the model $V[G_1]$ finds.

Finally, suppose that  $C(V, C_i)=C(V[G_0], C_i)$ holds for each $i\in j$, with the common value denoted by $D_i$. Then $\bigcup_iD_i$ is equal to both $C(V, C)$ and $C(V[G_0], C)$, and by the choice of the set $C$ it is equal to both $C(V, A)$ and $C(V[G_0], A)$. Thus, the latter two sets are equal as desired.
\end{proof}

\noindent It is important to see how the mutual Noetherian property of generic extensions plays with mutual genericity. This is the contents of the following proposition.

\begin{proposition}
\label{productproposition}
Let $n\geq 1$ be a number. Let $V[G_0], V[G_1]$ be generic extensions and $V[G_1]$ is $n$-Noetherian over $V[G_0]$. Suppose that $P_0\in V[G_0]$ and $P_1\in V[G_1]$ be posets and $H_0\subset P_0$ and $H_1\subset P_1$ be filters mutually generic over $V[G_0, G_1]$. Then $V[G_1][H_1]$ is $n$ Noetherian over $V[G_0][H_0]$.
\end{proposition}

\begin{proof}
Work in the model $V[G_0, G_1]$ and consider the poset $P_0\times P_1$. Let $X$ be a $K_\gs$ Polish space and $\mathcal{T}$ an analytic Noetherian topology on it of Krull dimension smaller than $n$, both in $V$. Let $p_0\in P_0$ and $p_1\in P_1$ be conditions and let $\tau_0, \tau_1$ be respective $P_0, P_1$-names in the models $V[G_0], V[G_1]$ such that $p_0\Vdash\tau_0$ is a $\mathcal{T}$-closed subset of $X$, $\tau_1\subset X$ is a set, and $\langle p_0, p_1\rangle\Vdash\tau_0= C(V[G_0][H_0], \tau_1)$; I must produce a ground model coded closed set $C\subset X$ such that $p_0\Vdash\tau_0=C$. 

Working in $V[G_1]$, form the closed set $A\subset X$ as $A=X\setminus \bigcup\{O\colon O\subset X$ is open and $p_1\Vdash O\cap\tau_1=0$. By the initial assumptions on the models $V[G_0]$ and $V[G_1]$, $C(V[G_0], A)=C(V, A)$ holds; write $C$ for the common value. Observe that $p_1\Vdash\tau_1\subset C$.  It will be enough to show that $p_0\Vdash \tau_0=C$.

Since $p_1\Vdash\tau_1\subset C$, the only way how the equality can fail is that there is a condition $p'_0\leq p_0$ forcing $\tau_0$ to be a proper subset of $C$. Working in $V[G_0]$, let $M_0$ be a countable elementary submodel of some large structure containing $\tau_0, C$, and $p'_0$. Let $h_0\subset P_0\cap M_0$ be a filter generic over the model $M_0$ and let $D=\tau_0/h_0$. This is a $\mathcal{T}$-closed set properly smaller than $C$, so $A\subseteq D$ fails. Thus, there must be a basic open set $O\subset X$ disjoint from $D$ which contains some element of the set $A$. By the definitions, this means that there is a condition $p''_0\leq p'_0$ in the filter $h_0$ which forces $\tau_0\cap O=0$, and a condition $p'_1\leq p_1$ which forces $\tau_1\cap O\neq 0$. This contradicts the initial assumptions on the conditions $p_0, p_1$.
\end{proof}

\begin{corollary}
Mutually generic extensions are mutually Noetherian.
\end{corollary}

\noindent The key technical point behind the present paper is that there are rather simple procedures for producing mutually Noetherian extensions of precisely calibrated dimension. The known examples are produced via a certain \emph{duplication} technique and the following proposition.

\begin{proposition}
\label{neatproposition}
Let $n\geq 2$ be a number and $\mathcal{T}$ be an analytic Noetherian topology on a $K_\gs$-Polish space $X$ of Krull dimension less than $n$. Let $\langle V[G_i]\colon i\in n\rangle$ be a tuple of forcing extensions such that for any sets $b_0, b_1\subset n$, $V[G_i\colon i\in b_0]\cap V[G_i\colon i\in b_1]=V[G_i\colon i\in b_0\cap b_1]$. For every irreducible $\mathcal{T}$-closed set $A\subset X$ there is $i\in n$ such that $C(V, A)=C(V[G_i], A)$.
\end{proposition}

\noindent The intersection condition on the tuples of generic extensions is satisfied for example for a mutually generic tuple by the product forcing theorem. The set $A$ does not have to belong to any of the models mentioned.

\begin{proof}
Suppose that the models $V[G_i]$ for $i\in n$ are given. For each set $b\subseteq n$ write $M_b=V[G_i\colon i\in b]$.

\begin{claim}
For every set $a\subset n$ and every closed set $C\subset X$ coded in $M_a$ there is an inclusion-smallest set $b\subset a$ such that $C$ is coded in $M_b$.
\end{claim}

\begin{proof}
Let $b_0, b_1\subset a$ be inclusion-minimal sets such that $C$ is coded in $M_{b_0}$ and $M_{b_1}$; I must show that $b_0=b_1$. Suppose towards a contradiction that the equality fails. The assumptions on the generic extensions then show that $C$ is coded in $M_{b_0\cap b_1}$, contradicting the minimal choice of both $b_0$ and $b_1$.
\end{proof}

\noindent Now, let $A\subset X$ be a nonempty irreducible $\mathcal{T}$-closed set. Let $C_b\subset X$ be the smallest $\mathcal{T}$-closed set coded in $M_b$ such that $A\subset C_b$. 

\begin{claim}
The sets $C_b$ are irreducible. $c\subseteq b$ implies $C_b\subseteq C_c$.
\end{claim}

\begin{proof}
For the first sentence, assume towards a contradiction that $C_b$ is not irreducible. By a Shoenfield absoluteness argument, $C_b$ is not irreducible in $M_b$ and one can express in $M_b$ the set $C_b$ as a union of $\mathcal{T}$-closed proper subsets: $C_b=\bigcup_{j\in i}D_j$. The irreducibility of the set $A$ then shows that there is $j\in i$ such that $A\subseteq D_j$, contradicting the minimal choice of $C_b$. The second sentence of the claim is immediate. 
\end{proof}

\noindent Now, suppose towards a contradiction that the conclusion of the proposition fails, i.e.\ $C_0\neq C_{\{i\}}$ for any $i\in n$.

\begin{claim}
For every nonempty set $a\subset n$ there is a subset $b\subset a$ of cardinality $|a|-1$ such that $C_a$ is not coded in $M_b$.
\end{claim}

\begin{proof}
Let $c\subset a$ be the inclusion-smallest set such that $C_a\in M_c$. The set $c$ is nonempty: otherwise, it would be the case that $C_0=C_a$, while for every index $i\in a$ it must be the case that $C_a\subseteq C_{\{i\}}$ while $C_0\not\subseteq C_{\{i\}}$. Now, choose any set $b\subset a$ of cardinality $|a|-1$ such that $c\not\subseteq b$; the set $b$ works by the minimal choice of $c$.
\end{proof}

\noindent Now, use the claim repeatedly to construct by downward recursion on $i\leq n$ sets $b_i\subset n$ such that for all $i\leq n$ $|b_i|=i$, the sets are linearly ordered with respect to inclusion, and $C_{b_i}$ is not coded in $M_{b_{i-1}}$. The sets $C_{b_i}$ for $i\leq n$ then form a sequence of nonempty irreducible sets strictly decreasing by inclusion. This contradicts the assumption on the Krull dimension of the topology $\mathcal{T}$.
\end{proof}

\noindent A number of interesting examples can now be produced via manipulation of Cohen forcings. If $X$ is a Polish space then the Cohen forcing $P_X$ is the poset of nonempty open subsets of $X$ ordered by inclusion. It adds a point $\dot x\in X$, the unique point in the intersection of open sets in the generic filter. It is not difficult to see \cite[Proposition 3.1.1]{z:geometric} that if $f\colon X\to Y$ is a continuous open map between Polish spaces then $P_X\Vdash\dot f(\dot x)$ is a point $P_Y$-generic over the ground model.

\begin{example}
\label{mainexample}
Let $n\geq 1$ be a number and $\eps>0$ be a positive real number. Let $Z=\{\langle z_0, z_1\rangle\in\mathbb{R}^n\times\mathbb{R}^n\colon d_(z_0, z_1)=\eps\}$, where $d$ is the Euclidean distance in $\mathbb{R}^n$. $Z$ is a closed subset of $\mathbb{R}^n\times\mathbb{R}^n$ equipped with the inherited Polish topology. Let $\langle z_0, z_1\rangle$ be a $P_{Z}$-generic pair over the ground model $V$. Then

\begin{enumerate}
\item both $z_0, z_1$ are generic elements of $\mathbb{R}^n$ over $V$;
\item the models $V[z_0]$ and $V[z_1]$ are mutually $n$-Noetherian.
\end{enumerate}
\end{example}

\begin{proof}
The first item follows from the fact that the projection function from $Z$ to any of the two coordinates is open. The second item is the heart of the matter. Suppose towards a contradiction that it fails. Then, by Proposition~\ref{enoughproposition}, in the ground model there must be a $K_\gs$ Polish space $X$, an analytic Noetherian topology $\mathcal{T}$ on $X$ of Krull dimension smaller than $n$, a $P_{\mathbb{R}^n}$-name $\tau$ for an irreducible $\mathcal{T}$-closed subset of $X$, and a condition $p\in P_{X_n}$ which forces $C(V, \tau/\dot x_1)\neq C(V[\dot x_0], \tau/\dot x_1)$.

Shrinking the condition $p$ if necessary, it is possible to find nonempty open sets $O_0, O_1\subset\mathbb{R}^n$ such that $p=(O_0\times O_1)\cap X_n$. It is a simple exercise in Euclidean geometry \cite[Claim 4.9]{z:krull} is to find nonempty open sets $O_{0i}\subset O_0$ for $i\in n$ such that for every tuple $\langle x_i\colon i\in n\in\prod_{i\in n}O_{in}$ there is a unique point $x_n$ in $O_1$ which is at distance $\eps$ from every point $x_i$ for $i\in n$. Let $Y=\{\langle x_i\colon i\in n+1\rangle\colon \forall i\in n x_i\in O_{0i}$, $x_n\in O_1$ and $d(x_i, x_n)=\eps\}$ This is a $G_\gd$-subset of $\mathbb{R}^n)^{n+1}$ at therefore Polish in the inherited topology. Consider the poset $P_Y$ of relatively open subsets of $Y$, and consider a tuple $\langle x_i\colon i\in n+1\rangle\in Y$ $P_Y$-generic over $V$.

\begin{claim}
\label{mutualclaim}
The tuple $\langle x_i\colon i\in n\rangle$ consists of mutually Cohen generic points of $\mathbb{R}^n$ over $V$.
\end{claim}

\begin{proof}
$Y$ is a graph of a continuous function from $\prod_iO_{0i}$ to $O_1$. The projection of a graph of a continuous function to its domain is an open map.
\end{proof}

\begin{claim}
\label{secondclaim}
For each $i\in n$, the pair $\langle x_i, x_n\rangle\in X_n$ is generic over $V$ for the poset $P_{X_n}$.
\end{claim}

\begin{proof}
The projection function from $Y$ to any pair of coordinates including $n$ is open from $Y$ to $Z$.
\end{proof}

\noindent Now, consider the set $\tau/x_n\subset X$. By Claim~\ref{mutualclaim} and Proposition~\ref{neatproposition}, there is a number $i\in n$ such that $C(V, \tau/x_n)=C(V[x_i], \tau/x_n)$. This, however, contradicts the initial assumptions on the name $\tau$ in view of Claim~\ref{secondclaim}.
\end{proof}

\section{A preservation theorem}

This section provides a preservation theorem connecting the Noetherian properties of generic extensions with an independence result. Recall that a Suslin forcing is a pair $\langle P, \leq\rangle$ such that $P$ is an analytic subset of some ambient Polish space, $\leq$ is a transitive relation on $P$ containing the diagonal, and $\leq$ and the incompatibility relations are both analytic subsets of $P^2$.

\begin{definition}
Let $P$ be a Suslin poset and $n\geq 1$ be a number.

\begin{enumerate}
\item A pair $\langle Q, \gs\rangle$ is $n$\emph{-Noetherian balanced} if $Q\Vdash\gs\in P$, and whenever $V[H_0], V[H_1]$ are mutually $n$-Noetherian extensions and in each there are respective filters $G_0, G_1\subset Q$ generic over $V$ and conditions $p_0\leq \gs/G_0$ and $p_1\leq\gs/G_1$, then $p_0, p_1$ are compatible in $P$. 
\item The poset $P$ is balanced of dimension characteristic $n$ if below every condition $p\in P$ there is an $n$-Noetherian balanced pair $\langle Q, \gs\rangle$ such that $Q\Vdash\gs\leq\check p$.
\end{enumerate}
\end{definition}

\begin{theorem}
\label{preservationtheorem}
In every generic extension of the choiceless Solovay model which is $\gs$-closed and cofinally $n$-Noetherian balanced, for every non-meager set $A\subset\mathbb{R}^n$ there is a real number $\eps(A)>0$ such that for every positive real number $\eps<\eps(A)$ there are two points in $A$ at distance $\eps$ from each other.
\end{theorem}

\begin{proof}
Let $\kappa$ be an inaccessible cardinal. Let $P$ be a $\gs$-closed Suslin forcing which is balanced of dimension characteristic $\geq n+1$ cofinally below $\kappa$. Let $W$ be the choiceless Solovay model derived from $\kappa$. Work in the model $W$. Let $p\in P$ be a condition and $\tau$ a $P$-name for a non-meager (or non-null) subset of $\mathbb{R}^n$. Towards a contradiction, assume that $p\Vdash\tau$ is a counterexample to the conclusion of the theorem. By the $\gs$-closure of the poset $P$ and Axiom of Dependent Choices in $W$, strengthening the condition $p$ if necessary, one can find a countable set $a\subset\mathbb{R}^+$ which has zero as an accumulation point and $p\Vdash\tau$ contains no two points whose distance belongs to $a$.

The name $\tau$ is definable from a ground model element and a real parameter $z\in\cantor$. Let $V[K]$ be an intermediate extension obtained by a poset of cardinality smaller than $\kappa$ such that $z, p, a\in V[K]$ in which the poset $P$ is $n+1$-Noetherian balanced. Work in $V[K]$. Let $\langle Q, \gs\rangle$ be an $n+1$-Noetherian balanced pair such that $Q\Vdash\gs\leq p$.

Let $R$ be the Cohen poset of nonempty open subsets of $\mathbb{R}^n$, adding a generic point $\dot x$. There must be conditions $q\in Q$, $r\in R$, and a poset $S$ of cardinality smaller than $\kappa$, a condition $s\in S$ and a $R\times Q\times S$-name $\eta$ for a condition stronger than $\gs$ such that

$$q\Vdash_Qr\Vdash_Rs\Vdash_S\coll(\gw, <\kappa)\Vdash \eta\Vdash_P\dot x\in\tau.$$ 

\noindent Otherwise, in the model $W$, for any generic filter $H\subset Q$ the condition $\gs/H\leq p$ would force in $P$ that $\tau$ is disjoint from the co-meager set of elements of $\mathbb{R}^n$ Cohen-generic over $V[K][H]$, contradicting the initial assumptions on $\tau$ and $p$. Let $\eps\in a$ be a real number smaller than the radius of some open ball which is a subset of $r$, and let $T$ be the poset for adding a generic pair of points in $\mathbb{R}^n$ of distance $\eps$ as in Example~\ref{mainexample}.

Work in the model $W$. Let $\langle x_0, x_1\rangle\in X_n$ be a pair $T$-generic over $V[K]$ below $r\times r$. Let $H_0, H_1\subset Q\times S$ be filters mutually generic over the model $V[K][x_0, x_1]$ meeting the conditions $q\in Q, s\in S$. Consider the conditions $p_0=\eta/x_0, H_0$ and $p_1=\eta/x_1, H_1$. Note that the points $x_0, x_1$ are $R$-generic over $V[K]$ and the models $V[K][x_0]$, $V[K][x_1]$ are $n+1$-mutually Noetherian by Example~\ref{mainexample}. The models $V[K][x_0][H_0]$ and $V[K][x_1][H_1]$ are $n+1$-mutually Noetherian by Proposition~\ref{productproposition}. The initial assumption on the balanced pair $\langle Q, \gs\rangle$ shows that the conditions $p_0, p_1$ are compatible. Their common lower bound forces the points $x_0, x_1$ into $\tau$, while their distance belongs to the set $a$. This contradicts the choice of the set $a$.
\end{proof}

\section{A coloring poset}

Finally, this section provides, for a given number $n\geq 1$ a definition of a Suslin forcing which, if used in the choiceless Solovay model, adds a countable coloring of the graph $\Gamma_n$ yet keeps the chromatic number of $\Gamma_{n+1}$ uncountable. In fact, there is a forcing which colors a much larger class of graphs on $\mathbb{R}^n$. This is recorded in the following theorem.

\begin{theorem}
\label{posettheorem}
Let $n\geq 1$ be a number. Let $\Gamma$ be a $\gs$-algebraic graph on $\mathbb{R}^n$ containing no perfect clique. There is a Suslin $\gs$-closed forcing $P$ such that

\begin{enumerate}
\item $P\Vdash$the union of the generic filter is a total $\Gamma$-coloring with countable range;
\item $P$ is $k, 2$-centered for every number $k\in\gw$;
\item under the Continuum Hypothesis, $P$ is $n+1$-Noetherian.
\end{enumerate}
\end{theorem}

\subsection{Preliminaries}

The proof of Theorem~\ref{posettheorem} requires a fair amount of familiarity with real algebraic geometry. The following remarks record some of the facts used. The theory of \emph{real closed fields} uses the ordering symbol, addition, multiplication, and $0, 1$ constants, and includes the axioms of ordered fields as well as axioms stating that every polynomial of odd degree has a root. A good reference for treatment of real closed fields is \cite[Section 3.3]{marker:book}

\begin{fact}
The theory of real closed fields admits quantifier elimination.
\end{fact}

\begin{corollary}
If $F\subset\mathbb{R}$ is a real closed subfield then $F$ is an elementary submodel of $\mathbb{R}$.
\end{corollary}

\noindent All of the consequences of these facts needed below concern algebraic subsets of Euclidean spaces.

\begin{corollary}
\label{c1}
Every algebraic set in a Euclidean space is either finite or uncountable.
\end{corollary}

\begin{proof}
This is true even for semi-algebraic sets. Let $n\geq 1$ be a number and $A\subset\mathbb{R}^n$ be a semi-algebraic set. Consider the projections $A_i$ of $A$ into each coordinate $i\in n$. If these sets are all finite, then $A$ is finite. Assume that some $A_i$ is infinite. Since it is definable by a quantifier-free formula in the language of real closed fields and it is infinite, it must contain a nonempty open interval. Thus, $A_i$ is uncountable and so is $A$.
\end{proof}

\noindent If $F\subset \mathbb{R}$ is a real closed subfield and $n\geq 1$ is a number, a set $A\subset\mathbb{R}^n$ is \emph{algebraic over $F$} if there is a polynomial $\phi(\bar u)$ of $n$ many variables and coefficients in $F$ such that $A=\{x\in \mathbb{R}^n\colon \phi(x)=0\}$. Several closure properties of the class of sets algebraic over $F$ will be needed. It is immediate that this class is closed under unions, intersections, and sections indexed by elements of $F$. The following closure properties are more delicate.

\begin{corollary}
\label{c2}
If $F\subset\mathbb{R}$ is a real closed field, $n, m\geq 1$, and $A\subseteq \mathbb{R}^m$, $B\subseteq\mathbb{R}^n\times\mathbb{R}^m$ are sets algebraic over $F$, then the set $C=\{x\in\mathbb{R}^n\colon \forall y\in A\ \langle x, y\rangle\in B\}$ is a set algebraic over $F$.
\end{corollary}

\begin{proof}
For each $y\in A$ let $C_y=\{x\in\mathbb{R}^n\colon \langle x, y\rangle\in B\}$; this is an algebraic set and $C=\bigcap_{y\in A}C_y$. By the Hilbert basis theorem, there is a finite set $a\subset A$ such that $C=\bigcap_{y\in a}C_y$. Let $k\in\gw$ be the cardinality of the set $a$. The existence of a set $b\subset A$ of cardinality $k$ such that $\forall x\in\mathbb{R}^m\ (\forall y\in b\ \langle x, y\rangle\in B)\to (\forall y\in A\ \langle x, y\rangle\in B)$ is a first-order statement in the language of real closed fields and as such it is reflected to $F$. Let $b\subset F$ be a witness that $F$ finds, and observe that $C=\bigcap_{y\in b}C_y$ is a set algebraic over $F$.
\end{proof}

\begin{corollary}
\label{c3}
Suppose that $F\subset\mathbb{R}$ is a real closed field and a set $A\subset\mathbb{R}^n$ is algebraic over $F$. If $A$ is reducible, then its irreducible composants are algebraic over $F$.
\end{corollary}

\begin{proof}
The decomposition of $A$ is well-known to be unique: among all ways of expressing the set $A$ as a union of finitely many algebraic sets none of which is covered by the others, it is the one in which there is the largest number of sets possible. If the number of composants is $m\in\gw$ and all of them are given by polynomials of degree at most $k$, the existence of such a decomposition is a first order statement in the language of real closed fields, which is then reflected by the field $F$.
\end{proof}

\begin{corollary}
\label{c4}
A finite set $A\subset\mathbb{R}^n$ algebraic over $F$ is a subset of $F^n$.
\end{corollary}

\begin{proof}
Let $k$ be the cardinality of $A$. Existence of a list of $k$ points exhausting the set $A$ is a first-order statement of the language of real closed fields, and therefore reflected by $F$.
\end{proof}

\noindent Finally, there is a proposition about algebraic sets which is needed in one critical spot of the proof of Theorem~\ref{posettheorem}. While it seems to be either folkloric or otherwise well-known, I did not find a good reference for it and I include the proof.

\begin{proposition}
\label{folkproposition}
Let $n\geq 1$ be a number. Let $A\subset\mathbb{R}^n$ be an irreducible algebraic set, and $A\subseteq\bigcup_iB_i$ for some algebraic sets $B_i$ for $i\in\gw$. Then there is an index $i\in\gw$ such that $A\subseteq B_i$.
\end{proposition}

\begin{proof}
This does not follow from a straightforward application of the Baire category theorem to the closed set $A$, since irreducible algebraic sets in Euclidean spaces may contain for example points isolated in the sense of the Euclidean topology. It is necessary to apply the Baire category theorem to a certain relatively open subset of $A$. Below, the dimension of an irreducible algebraic set is its Krull dimension, and the dimension of an arbitrary algebraic set is the maximal Krull dimension of its irreducible components.

Without loss, assume that $A\neq 0$. Consider the set $\regg(A)\subseteq A$ consisting of all non-singular points of $A$--these are the points where the dimension of the Zariski tangent space to $A$ is exactly equal to $\dim(A)$. It is well-known \cite[Proposition 3.3.14] {bochnak:real} that $\sing(A)=A\setminus\regg(A)$ is an algebraic subset of $A$ of dimension strictly smaller than $\dim(A)$; in particular, $\regg(A)\neq 0$ and $\regg(A)$ is a relatively open subset of $A$.

\begin{claim}
Suppose that $B$ is a proper algebraic subset of $A$. Then $B$ is nowhere dense in the set $\regg(A)$ in the topology inherited from $\mathbb{R}^n$.
\end{claim}

\begin{proof}
Since every algebraic set is a finite union of its irreducible components, it is enough to prove the claim for irreducible $B$. Since algebraic sets are closed, the statement is equivalent to showing that $\regg(A)\setminus B$ is dense in $\regg(A)$. The argument proceeds by induction on $m=\dim(B)$, which is necessarily smaller than $\dim(A)$. The base step $m=1$ is subsumed in the induction step. For the induction step, suppose that the statement is known for some number $m$, $\dim(B)=m+1$, and $O\subset\regg(A)$ is a nonempty relatively open set; I need to produce a point in $O\setminus B$. $\sing(B)$ is an algebraic set of dimension smaller than that of $\dim(B)$, so by the induction hypothesis it is possible to shrink $O$ if necessary to contain no points in $\sing(B)$. Now suppose towards a contradiction that $O\cap B=O\cap A$. This is a relatively open set of non-singular points of both $A$ and $B$. Therefore, at any point in it, its $C^\infty$-tangent space coincides with the Zariski tangent space for both $A$ and $B$, and should therefore have dimension equal to both $\dim(A)$ and $\dim(B)$. Since $\dim(B)<\dim(A)$, this is impossible.
\end{proof}

The proposition now follows by an application of the Baire category theorem to the set $\regg(A)$, which is Polish in the topology inherited from $\mathbb{R}^n$.
\end{proof}

\subsection{Proof of Theorem~\ref{posettheorem}}

To set up notation for the proof, fix the number $n\geq 1$ and write $X=\mathbb{R}^n$. Let $\{\phi_m\colon m\in\gw\}$ be a countable family of polynomials generating the graph $\Gamma$. For brevity of notation, assume that each of the polynomials $\phi_m$ is symmetric, i.e.\  $\phi_m(u, v)=\phi_m(v, u)$. Assume also that the parameters of the polynomials are all integers. For a set $A\subset X$ and a number $m\in\gw$, write $\phi_m(A)=\{x\in X\colon\forall y\in A\ \phi_m(x, y)=0\}$. By the Hilbert basis theorem, there is a finite set $A'\subset A$ such that $\phi_m(A)=\phi_m(A')$; thus, the set $\phi_m(A)$ is algebraic. The set of colors $\colors$ consists of all pairs $\langle O, b\rangle$ where $O$ is an open ball in $X$ with rational center and rational radius and $b\subset\gw$ is finite.

\begin{definition}
\label{posetdefinition}
The coloring poset $P$ consists of all partial $\Gamma$-colorings $p$ whose range is a subset of $\colors$ and

\begin{enumerate}
\item[(A)] for some countable real closed field $\supp(p)\subset\mathbb{R}$, $\dom(p)=\supp(p)^n$;
\item[(B)] for every uncountable irreducible set $A\subset X$ algebraic over $\supp(p)$ there is a finite set $b(p, A)\subset \gw$ such that every color $\langle O, b\rangle\in\colors$ with $O\cap A\neq 0$, $b(p, A)\subset b$, and $\forall m\in b\ O\cap \phi_m(A)=0$ is attained infinitely many times on $A$.
\end{enumerate}

\noindent  The ordering on $P$ is defined by $q\leq p$ if $p\subseteq q$ and for every set $A\subset X$ algebraic over $\supp(p)$, $p''A=q''A$ holds. 
\end{definition}

\noindent The properties of the poset $P$ are verified in a long sequence of propositions.

\begin{proposition}
\label{sigmaproposition}
$\leq$ is a $\gs$-closed preordering.
\end{proposition}

\begin{proof}
The transitivity is immediate. For the $\gs$-closure, if $p_i\colon i\in\gw\rangle$ is a descending chain of conditions, $\bigcup_ip_i$ is the common lower bound.
\end{proof}

\noindent The most important part of the proof is a precise and generous characterization of compatibility of conditions in the poset $P_n$.

\begin{proposition}
\label{compatibilityproposition}
Let $a\subset P_n$ be a finite set of conditions. The following are equivalent:

\begin{enumerate}
\item $a$ has a common lower bound;
\item for every $x\in X$, $a$ has a common lower bound containing $x$ in its domain;
\item $q=\bigcup a$ is a function and a $\Gamma$-coloring, and for every $p\in a$ and every set $A\subset X$ algebraic over $\supp(p)$, $p''A=q''A$.
\end{enumerate}
\end{proposition}

\begin{proof}
Clearly, (2) implies (1), which implies (3) by the definition of the ordering on the poset $P_n$. To show that (3) implies (2), assume that (3) holds and $x\in X$ is an arbitrary point. Let $M$ be a countable elementary submodel of a large structure containing the set $a$ and the point $x$, and let $F=M\cap\mathbb{R}$. I will construct a lower bound $r$ of $a$ such that $\supp(r)=F$.

The construction of $r$ is a demanding counting argument; it is necessary to set some notation. Let $d=F^n\setminus\bigcup_{p\in a}\dom(p)$. For every point $y\in d$ and every condition $p\in a$ write $B(p, y)\subset X$ for the inclusion-smallest set algebraic over $F$ containing $y$ as an element. Similarly, for every irreducible set $A\subset X$ algebraic over $F$ and every condition $p\in a$ write $B(p, A)$ for the smallest set which is algebraic over $\supp(p)$ containing $A$ as a subset. Note that the set $B(p, A)$ actually exists, since there are no infinite sequences of algebraic sets strictly decreasing under inclusion by the Hilbert basis theorem.  In addition, the set $B(p, A)$ is irreducible: if not, its irreducible composants would be algebraic over $\supp(p)$ by Corollary~\ref{c3}, one of these composants would have to cover $A$ by the irreducibility of $A$, and the minimal choice of $B(p, A)$ would be violated.

The condition $r$ is constructed by finite approximations of the following kind. Call a pair $\langle f, g\rangle$ an \emph{approximation} if 

\begin{itemize}
\item[(i)] $f\colon d\to\colors$ is a finite partial $\Gamma$ coloring and for each $y\in\dom(r)$ and every $p\in a$, the color $r(y)$ is attained infinitely many times in $p$ on $B(p, y)$;
\item[(ii)] $g$ is a function whose domain consists of finitely many uncountable irreducible subsets of $X$ algebraic over $F$, and for each $A\in\dom(g)$, $g(A)\subset\gw$ is a finite superset of $\bigcup_{p\in a}b(p, B(p, A))$;
\item[(iii)] if $y\in\dom(f)$ and $A\in\dom(g)$ are a point and a set such that for some number $m\in\gw$ $\forall z\in A\ \phi_m(y, z)=0$ holds, then writing $f(y)=\langle O, c\rangle$ some such number $m$ belongs to $c\cup g(A)$ and either $g(A)\not\subseteq c$ or $O$ contains some point $y'$ such that $\forall z\in A\ \phi_m(y' z)$;
\item[(iv)] if $A_0, A_1\in\dom(g)$ are sets such that for some $m\in\gw$, $\forall y_0\in A_0\ \forall y_1\in A_1\ \phi_m(y_0, y_1)=0$, then some such number $m$ belongs to $g(A_0)\cup g(A_1)$.
\end{itemize}

\noindent Approximations are ordered by coordinatewise reverse extension. It is necessary to show that approximations can be suitably extended; this is the purpose of the following claims.

\begin{claim}
Let $\langle f, g\rangle$ be an approximation and $y\in d$ be a point. Then there is a color $\langle O, c\rangle\in\colors$ such that $\langle f\cup \{\langle y, \langle O, c\rangle\rangle\}, g\rangle$ is an approximation.
\end{claim}

\begin{proof}
Without loss, assume that $y\notin\dom(f)$. Note that $y$ is an accumulation point of every set $B(p, y)$ for $p\in a$ as $y\notin\supp(p)^n$. Find a finite set $c\subset\gw$ which includes $\bigcup_{p\in a}b(p, B(p, x_i))$, and for every set $A\in\dom(g)$, if there is $m$ such that $\forall z\in A\ \phi_m(y,z)=0$ then some such $m$ is in $c$. Now, for every $m\in\gw$ and every $p\in a$, consider the set $C(p, m)=B(p, y)\cap \phi_m(B(p, y))$. The set $C(p, m)$ is algebraic over $\supp(p)$ by Corollary~\ref{c2}. It is a $\Gamma$-anticlique, therefore it is not uncountable by the assumptions on $\Gamma$, therefore it must be finite by Corollary~\ref{c1},  therefore it is a subset of $\dom(p)$ by Corollary~\ref{c4}, therefore it does not contain the point $y$. Find a basic open neighborhood $O$ of $y$ which is disjoint from the sets $C(p, m)$ for $m\in c$ and $p\in a$, and such that $O$ is not used in range of $f$. The color $\langle O, c\rangle$ is then as required.
\end{proof}

\begin{claim}
Let $\langle f, g\rangle$ be an approximation and $A\subset X$ an uncountable irreducible set algebraic over $F$. Then there is a finite set $b\subset\gw$ such that $\langle f, g\cup \{\langle A, b\rangle\}\rangle$ is an approximation.
\end{claim}

\begin{proof}
Without loss, assume that $A\notin\dom(g)$. Just let $b$ be any superset of all the sets $b(p, B(p, A))$ which in addition satisfies the following conditions. If $A'\in\dom(g)$ is such that for some $m\in\gw$, $\forall y\in A'\ \forall z\in A\ \phi_m(y, z)=0$, then some such number $m$ belongs to $b$.  If $y\in\dom(f)$ such that for some number $m\in\gw$ $\forall z\in A\ \phi_m(y, z)=0$ holds, then writing $f(y)=\langle O, c\rangle$ some such number $m$ belongs to $c\cup g(A)$ and $b\not\subseteq c$. The set $b$ is then as required.
\end{proof}

\begin{claim}
Let $\langle f, g\rangle$ be an approximation, let $A\in\dom(g)$, and let $\langle O, c\rangle$ be a color such that $g(A)\subseteq c$, $O\cap A\neq 0$, and $\forall m\in c\ O\cap\phi_m(A)=0$. Then there is a point $z\in A\setminus\dom(f)$ such that $\langle f\cup\{\langle z, \langle O, c\rangle\rangle\}, g\rangle$ is an approximation.
\end{claim}

\begin{proof}
Consider the countable collection of algebraic sets including the following:

\begin{itemize}
\item[(a)] singletons in $\dom(q)\cup\dom(f)$;
\item[(b)] sets algebraic in $\supp(p)$ for $p\in a$;  
\item[(c)] for every number $m\in\gw$ and for every point $y\in\dom(f)$, the set $\{w\in X\colon \phi_m(z, w)=0\}$;
\item[(d)] for every number $m\in\gw$ and every set $A'\in\dom(g)$, the set $\phi_m(A')$.
\end{itemize}

\noindent Use Proposition~\ref{folkproposition} to find a point $z\in A$ which belongs only to those algebraic sets in this collection which are supersets of $A$. By elementarity of the model $M$, such a point $z$ can be found in $A\cap M$. Now, work to show that the point $z$ is as required. Validity of items (ii) and (iv) in the definition of an approximation remains untouched as they do not refer to any points in the domain of the first coordinate.

For item (i), argue that  $f\cup\{\langle z, \langle O, c\rangle\rangle\}$ is a $\Gamma$-coloring. Suppose that $y\in\dom(f)$ is a point such that $y\Gamma z$ holds, and find a number $m\in\gw$ such that $\phi_m(y, z)$ holds. Write $f(y)=\langle O', c'\rangle$. By the choice of the point $z$, it must be the case that $\forall w\in A\ \phi_m(y, w)=0$ holds. By item (iii) of approximation $\langle f, g\rangle$, some such a number $m$ belongs to $c'\cup g(a)\subseteq c'\cup c$ and either $g(A)\not\subseteq c'$ or there is a point $y'\in O'$ such that $\forall w\in A\ \phi_m(y', w)=0$ holds. In the former case, $c\neq c'$ and the colors of $y, z$ are distinct. In the latter case, either $c\neq c'$ and the colors are distinct, or $c=c'$ holds. In the latter case, $m\in c$ holds and $O$ does not contain the point $y'\in\phi_m(A)$, which means that $O\neq O'$ and the colors are distinct again.

To verify that $\langle O, c\rangle$ belongs to every set $p''B(p, z)$ for $p\in a$, note that $B(p, z)=B(p, A)$ holds by item (b) of the choice of the point $z$. It follows that $B(p, z)\subseteq g(A)\subseteq c$ holds by item (ii) applied to $\langle f, g\rangle$ and for every $m\in b(B(p, z))$ $O\cap\phi_m(A)=0$. Since $A\subseteq B(p, A)$, $\phi_m(B(p, A)\subseteq \phi_m(A)$ holds and $O\cap \phi_m(B(p, z))=0$. By item (B) of Definition~\ref{posetdefinition}, $\langle O,c\rangle\in p''B(p, z)$ as required.

For item (iii), suppose that $A'\in\dom(g)$ is a set such that for some number $m\in\gw$ $\forall w\in A'\ \phi_m(z, w)=0$ holds. Item (d) of the choice of the point $z$ then shows that $A\subset\phi_m(A')$. Some such number $m$ must belong to $g(A)\cup g(A')$ by (iv) of the approximation $\langle f, g\rangle$, and it follows that it belongs to $c\cup g(A')$. Since $O\cap A\neq 0$, there is a point $z'\in O$ such that $\forall w\in A'\ \phi_m(z, w)=0$ holds as required.
\end{proof}

With the three extension claims in hand, a simple bookkeeping argument produces a descending sequence $\langle f_i, g_i\colon i\in\gw\rangle$ such that, writing $f=\bigcup_i f_i$ and $g=\bigcup_ig_i$, $\dom(f)=d$, $\dom(g)$ contains all uncountable irreducible sets algebraic in $F$, and for every set $A\in\dom(g)$ every color $\langle O, c\rangle$ such that $g(A)\subseteq c$ and $O\cap A\neq 0$ and $O\cap \phi_m(A)=0$ for all $m\in c$, is attained infinitely many times in $f$ on $A$.

In the end of the recursion, let $r=q\cup f$ and argue that $r\in P$ is a common lower bound of the conditions in the set $a$. The argument is the conjunction of the following three claims.

\begin{claim}
$r$ is a condition in $P$.
\end{claim}

\begin{proof} 
To see Definition~\ref{posetdefinition}(A) , suppose that $y_0, y_1$ are $\Gamma$-connected points in $\dom(r)$. If neither of them belongs to the set $d$, then they receive distinct colors since $q$ is a $\Gamma$-coloring. If both of them belong to the set $d$, then they receive distinct colors since $f$ is a $\Gamma$-coloring. Finally, if $y_0\in\dom(p)$ for some $p\in a$ and $y_1\in d$, then for some number $m\in\gw$, $\phi_m(y_0, y_1)=0$. Then $B(p, y_1)$ is a subset of the set $\{y\in X\colon \phi_m(y_0, y)=0\}$ and by item (i) of the recursion hypothesis, there is $y'_1\in\dom(p)\cap B(p, y_1)$ distinct from $y_0$ such that $r(y_1)=p(y'_1)$. Since $p$ is a $\Gamma$-coloring and $y_0\mathrel\Gamma y'_1$, this means that $y_0, y_1$ receive distinct colors in this configuration as well.

To see Definition~\ref{posetdefinition}(B), suppose that $A\subset X$ is an irreducible set algebraic over $F$. It is clear from the bookkeeping demands that the set $g(A)$ witnesses  Definition~\ref{posetdefinition}(B) for $r, A$. 
\end{proof}

\begin{claim}
For every condition $p\in a$, $r\leq p$ holds.
\end{claim}

\begin{proof}
Let $A\subset X$ be a set algebraic over $\supp(p)$; I must show that $r''A=p''A$. This, however, follows immediately from item (i) of the definition of approximation and the assumption (3).
\end{proof}

\noindent The proof of the proposition is complete.
\end{proof}

\begin{corollary}
The poset $P$ is nonempty and Suslin.
\end{corollary}

\begin{proof}
It is clear that the $P$ and the ordering on it are both Borel in a suitable space. Proposition~\ref{compatibilityproposition} shows that the incompatibility relation is Borel as well. Nonemptiness follows from the proposition applied to an empty set of conditions.
\end{proof}

\begin{corollary}
For every number $m\in\gw$, the poset $P$ is $m, 2$-centered.
\end{corollary}

\noindent This is to say that if $a$ is a collection of cardinality $m$ consisting of pairwise compatible conditions, then it has a common lower bound. This is immediate from Proposition~\ref{compatibilityproposition}.

\begin{corollary}
\label{coloringcorollary}
The poset $P$ forces the union of its generic filter to be a total $\Gamma$-coloring.
\end{corollary}

\noindent By a genericity argument, this is to say that for each point $x\in X$ and each condition $p\in P_n$, there is $q\leq p$ such that $x\in\dom(q)$. This is immediate from Proposition~\ref{compatibilityproposition} applied to the singleton $\{p\}$.

\begin{corollary}
\label{balancecorollary}
Under the Continuum Hypothesis, the poset $P$ is $n+1$-Noetherian balanced.
\end{corollary}

\noindent Unlike the usual theorems in \cite{z:geometric}, I do not bother to produce an exact classification of balanced virtual conditions, which in this case is complicated and brings nothing new to the proofs. I do not know if the CH assumption can be dropped.

\begin{proof}
Let $p\in P_n$ be a condition. Let $\langle x_\ga\colon\ga\in\gw_1\rangle$ be an enumeration of $\mathbb{R}$. By recursion on $\ga\in\gw_1$ build conditions $p_\ga\in P_n$ so that
$p=p_0$, $p_{\ga+1}\leq p_\ga$ is a condition such that $x_\ga\in\supp(p_{\ga+1})$, and $p_\ga=\bigcup_{\gb\in\ga}p_\gb$ for limit ordinals $\ga$. The function $c=\bigcup_\ga p_\ga$ is a total $\Gamma$-coloring which satisfies Definition~\ref{posetdefinition}(B). It is clear that $\coll(\gw, \mathbb{R})\Vdash\check c\leq p$ is a condition in the poset $P$. By Proposition~\ref{compatibilityproposition}, it will be enough to show that the pair $\langle\coll(\gw, \mathbb{R}), \check c\rangle$ is $n+1$-Noetherian balanced.

To do this, suppose that $V[G_0], V[G_1]$ are mutually $n+1$-Noetherian extensions. Suppose that $p_0, p_1\in P_n$ are conditions in the respective models stronger than $c$; I must show that $p_0, p_1$ are compatible. This will be done using Proposition~\ref{compatibilityproposition}.

To show that $p_0\cup p_1$ is a $\Gamma$-coloring, suppose that $x_0\in\dom(p_0)$ and $x_1\in\dom(p_1)$ are $\Gamma$-related points. Let $m\in\gw$ be a number such that $\phi_m( x_0, x_1)=0$. If both points $x_0, x_1$ belong to one and the same condition, then they must receive distinct colors as that condition is a $\Gamma$ coloring. The only remaining configuration is that $x_0\in V[G_0]\setminus V$ and $x_1\in V[G_1]\setminus V$. Note that $B(c, x_1)=B(p_0, x_1)\subset\{y\in X\colon \phi_m(y_0, x)=0\}$ holds by the Noetherian assumption on the models $V[G_0]$ and $V[G_1]$. Since $p_1\leq c$ holds, there must be a point $x'_1\in\dom(c)\cap B(c, x_1)$ distinct from $x_0$ such that $p_1(x_1)=c(x'_1)$. By the irreflexivity assumption on the polynomial $p_m$, it must be the case that $x_0\neq x'_1$. Since $p_1$ is a $\Gamma$ coloring, it follows that $p_0(x_0)\neq c(x'_1)=p_1(x_1)$ holds as desired.

Finally, let $A\subset X$ be a set algebraic over $\supp(p_0)$ and $x\in A\cap\dom(p_1)$ be a point; I must find a point $x'\in A\cap\dom(p_0)$ which receives the same color as $x$. Use the  Noetherian assumption to see that $B(p_0, x)=B(c, x)\subseteq A$. Use the $p_1\leq c$ assumption to find a point $x'\in B(c, x)$ such that $c(x')=p_1(x)$. The point $x'$ works as desired.
\end{proof}

Finally, the stage has been set to prove Theorem~\ref{maintheorem2}. Let $n\geq 1$ be a number. Let $\Gamma$ be a graph on $\mathbb{R}^n$ without a perfect clique. Let $\kappa$ be an inaccessible cardinal and let $W$ be the choiceless Solovay model derived from $\kappa$; $W$ is a model of ZF+DC \cite[Theorem 26.14]{jech:newset}. Let $P$ be the coloring poset isolated in Theorem~\ref{posettheorem}. Let $G\subset P$ be a filter generic over $W$. The model $W[G]$ is a $\gs$-closed extension of $W$ by Proposition~\ref{sigmaproposition} and as such it is a model of ZF+DC. In addition, $\bigcup G$ is a total coloring of the graph $\Gamma$ with countably many colors by Corollary~\ref{coloringcorollary}. By Corollary~\ref{balancecorollary} and Theorem~\ref{preservationtheorem}, in the model $W[G]$ every nonmeager subset of $\mathbb{R}^{n+1}$ contains point of arbitrary small enough distance. By the Baire category theorem, it must be the case that the chromatic number of $\Gamma_{n+1}$ is uncountable in $W[G]$ as required. Thus, the model $W[G]$ witnesses the conclusion of Theorem~\ref{maintheorem}; in case that $\Gamma=\Gamma_n$, it witnesses the conclusion of Theorem~\ref{maintheorem}.

\section{Declarations}

\subsection{Funding}

No funding was received to assist with the preparation of this manuscript.

\subsection{Employment}

The author is employed by University of Florida. He anticipates employment by the Czech Academy of Sciences.

\subsection{Conflicts of interest/Competing interests}

The author has no relevant financial or non-financial interests to disclose.

\subsection{Data availability}

This paper has no associated datasets.

\bibliographystyle{plain} 
\bibliography{odkazy,zapletal,shelah}

\end{document}